\documentclass[12pt,a4paper]{amsart}         
\usepackage[margin=3cm]{geometry}

\usepackage{amsmath,amssymb,amsthm}
\usepackage{hyperref}
\usepackage{enumerate}

\usepackage[numbers]{natbib}

\usepackage{lineno}
\title{Enumeration of Nilpotent Loops by Orbit Counting}
\author{Fariha Iftikhar}
\address{Institute of Mathematics \\
        Budapest University of Technology and Economics\\
        M\H{u}egyetem rkp 3\\
        H-1111 Budapest, Hungary}
\email{fariha.iftikhar@edu.bme.hu}
\author{G\'abor P. Nagy}
\address{Bolyai Institute \\
        University of Szeged \\
        Aradi v\'ertan\'uk tere 1\\
        H-6720 Szeged, Hungary}
\thanks{This work was supported by the National Research, Development and Innovation Fund of the Ministry for Innovation and Technology of Hungary under grant SNN 152582.}
\email{nagyg@math.u-szeged.hu}

\date{\today}

\newtheorem{theorem}{Theorem}[section]
\newtheorem{lemma}[theorem]{Lemma}
\newtheorem{proposition}[theorem]{Proposition}

\theoremstyle{definition}
\newtheorem{definition}[theorem]{Definition}
\newtheorem{remark}[theorem]{Remark}

\newcommand{\boldQ}{\mathbf{Q}}
\newcommand{\boldO}{\mathbf{O}}
\newcommand{\boldQexact}{\mathbf{Q}_{\mathrm{exact}}}
\newcommand{\boldQnonexact}{\mathbf{Q}_{\mathrm{nonexact}}}
\newcommand{\NPL}{\mathbf{NPL}}

\begin{document}

\begin{abstract}
% We study central extensions of a nilpotent loop $F$ by an elementary abelian $p$-group $A$ via normalized cocycles $\theta \in C(F, A)$. Instead of working in the quotient space of $C(F, A)$ modulo the coboundaries $B(F, A)$, we introduce an affine automorphism group $\mathrm{Aut}^{\!*}(F, A)$ combining the natural $\mathrm{Aut}(F)\times \mathrm{Aut}(A)$ action with coboundary translations. We prove that the $\mathrm{Aut}^{\!*}(F, A)$-orbits on $C(F, A)$ are in bijection with the isomorphism classes of loops $L$ with a fixed central subgroup $A$ such that $F\cong L/A$. In the enumeration of nilpotent loops, we characterize nonexact cocycles via linear centrality and order conditions, and derive an efficient orbit-counting procedure for exact extensions. The method matches existing enumerations for $|L|<24$.

We study central extensions of nilpotent loops by elementary abelian $p$-groups using normalized cocycles. By introducing an affine automorphism group acting on the full cocycle space, we obtain a direct correspondence between affine orbits and isomorphism classes of central extensions. This framework yields an efficient orbit-counting method for enumerating nilpotent loops. We reproduce the known results for orders less than 24, and enumerate the nilpotent loops of order 24 with center of size at least 3.
\end{abstract}

\maketitle

%==================%=======================%

\section{Introduction}

The classification of finite algebraic structures up to isomorphism constitutes a fundamental problem in algebra and combinatorics. Loops and Latin squares represent two prominent families of such structures, for which complete classifications are currently available only for comparatively small orders. From a combinatorial standpoint, enumerating these objects raises numerous unresolved questions, primarily due to the explosive growth in the number of distinct configurations and the increasing computational complexity of isomorphism testing. A common approach to classification problems is to reduce the original task to the analysis of extensions of simpler underlying structures; see \cite{nagypeter2019nuclear, SemaniSinovaStanovsky2023three}. In the setting of group theory, this principle is formalized by extension theory and group cohomology \cite{schreier1926erweiterung}. Specifically, a central extension of a group $G$ by an abelian group $A$ can be described in terms of $2$-cocycles and is classified up to equivalence by the second cohomology group $H^2(G, A)$. The cocycle-based description of loop extensions naturally embeds into the more general framework of the theory of loops and quasigroups \cite{bruck1971survey,pflugfelder1990quasigroups}. 

Daly and Vojt\v{e}chovsk\'y \cite{DalyVojtechovsky2009} initiated a systematic classification of nilpotent loops based on cocycle-defined central extensions and the corresponding cohomological framework. Their methodology yields a complete enumeration of all nilpotent loops of order less than $24$. Moreover, Daly and Vojt\v{e}chovsk\'y proved an explicit formula for the number of nilpotent loops of order $2q$, where $q$ is an odd prime, based on classifying central extensions of $\mathbb{Z}_q$ by $\mathbb{Z}_2$ up to automorphism-equivalence; see \cite[Theorem 7.1]{DalyVojtechovsky2009}. Building on this result, Clavier \cite{clavier2012enumeration} derived a formula for counting the isotopy classes of nilpotent loops of order $2q$. 

At the 2nd Mile High Conference on Nonassociative Mathematics (Denver, 2009), Vojt\v{e}chovsk\'y \cite{Vojtechovsky2009_ProblemNilpotentLoops24} formulated the open problem of determining, up to isomorphism, the number of nilpotent loops of order $24$. This question remained unresolved for more than a decade, primarily due to the substantial computational complexity of analyzing cocycles arising from central extensions of loops of order $12$. 

In the present work, we resolve this problem for loops of order $24$ whose center has size at least $3$. We anticipate that our methodology can be further developed to enumerate all nilpotent loops of order $24$.

The objective of this paper is to introduce an alternative formulation of the equivalence relation employed in cohomological classification. 
Rather than working with the quotient $H(F, A) = C(F, A) / B(F, A)$, we describe the relevant equivalence directly as the orbit relation arising from an affine group action on the entire cocycle space $C(F, A)$. 
This group is generated by the linear action of $\mathrm{Aut}(F)\times \mathrm{Aut}(A)$ together with translations of the form $\theta \mapsto \theta + \partial \tau$, where $\tau : F \to A$ satisfies $\tau(1) = 0$ and $\partial\tau(x,y)=\tau(xy)-\tau(x)-\tau(y)$.  
The resulting matrix group acts on a finite-dimensional vector space and is particularly amenable to implementation in GAP and to explicit orbit computation.

Our approach to determining the number of nilpotent loops of order $n$ is based on the following principal steps:

\begin{enumerate}
\item Introduce loop cocycles for \textbf{abstract central loop extensions} and certain restricted subclasses.
\item Define a group acting affinely on the affine spaces of loop cocycles so that orbits correspond to isomorphic extensions.
\item Count these orbits to determine the number of extensions up to isomorphy.
\item Apply inclusion-exclusion to obtain the number of nilpotent loops up to isomorphy.
\end{enumerate}

The structure of our paper is as follows. In Section~2, we review central extensions of loops and introduce the notion of exact, nonexact, and stretched extensions. In addition, we reduce the problem of enumerating nilpotent loops to that of counting exact extensions. Section~3 recalls the cocycle-based model for central extensions, together with cocycles and coboundary spaces. In Section~4, we introduce the affine automorphism group $\mathrm{Aut}^{\!*}(F, A)$, construct its natural affine action on the space of cocycles. We then prove that $\mathrm{Aut}^{\!*}(F, A)$-orbits on $C(F, A)$  are in bijection with the isomorphism classes of central extensions of $F$ by $A$. In Section~5, we derive an orbit counting procedure by Burnside's lemma and explain how fixed point counts reduce to linear algebra over finite fields. In Section~6, we develop the restriction method for efficiently computing exact extensions. We also introduce restricted cocycle subspaces and prove tameness conditions that justify the orbit counting problem on smaller subspaces. Section~7 reports the computational experiments that reproduce the known enumerations for orders less than $24$  and provides a performance comparison. In Section~8, we apply the method to the enumeration of nilpotent loops of order $24$ with center of size at least $3$.

\section{Central extensions of loops} 

A nonempty set $Q$ equipped with a binary operation $\cdot$ is called a \textit{loop} if there exists a neutral element $1 \in Q$ such that $1\cdot x = x\cdot 1 = x$ for all $x \in Q$, and if, for every $x \in Q$, the left and right translation maps
\[
L_x \colon Q \to Q,\quad L_x(y) = x\cdot y,\qquad
R_x \colon Q \to Q,\quad R_x(y) = y\cdot x
\]
are bijective. Finite loops are in one-to-one correspondence with normalized Latin squares, and groups are precisely those loops for which the binary operation is associative. A subloop $K \leq Q$ is said to be \textit{normal} in $Q$ if the conditions
\[
xK = Kx,\quad x(yK) = (xy)K,\quad x(Ky) = (xK)y,\quad (Kx)y = K(xy)
\]
hold for all $x,y \in Q$. For $ \ K\triangleleft Q$, the factor loop $Q/K$ is defined analogously to the factor group construction in group theory. If $F \cong Q/K$, then $Q$ is referred to as an \textit{extension} of $F$ by $K$.

The center $Z(Q)$ of a loop $Q$ is defined as the set of all elements that both commute and associate with every element of $Q$. Any subloop $A$ of $Z(Q)$ is necessarily a normal subloop of $Q$. For a prime number $p$, we denote by $Z_p(Q)$ the subgroup of the center $Z(Q)$ consisting of all elements whose orders divide $p$. Throughout this paper, $A$ will denote an abelian group, written additively. If $A$ is an elementary abelian $p$-group, then it can be (and it will be) naturally identified with a vector space of dimension $d$ over the finite field $\mathbb{F}_{p}$. An extension of a loop $F$ by a central subloop $A$ is referred to as a \textit{central extension}.

For a loop $L$, we denote by $[L]$ the isomorphism class of $L$. If $A \leq L$ is a subloop, we write $[L,A]$ for the isomorphism class of the pairs $(L,A)$, where an isomorphism between $(L_1,A_1)$ and $(L_2,A_2)$ is defined to be a loop isomorphism $f : L_1 \to L_2$ satisfying $f(A_1) = A_2$. For a loop $F$ and abelian group $A$, we use the notation
\begin{align} \label{eq:Qbold}
\boldQ(F,A) &= \bigl\{[L,A]\mid A \leq Z(L),\ L/A \cong F\bigr\}.
\end{align}
If $A=C_p^d$ is an elementary abelian $p$-group of order $p^d$, then we write $\boldQ(F,p^d) = \boldQ(F,A)$.

\subsection{Exact and nonexact extensions}
Daly and Vojt\v{e}chovsk\'y \cite{DalyVojtechovsky2009} have already noticed that using linear algebra, it is relatively easy to classify nilpotent loops whose center has prime order, but dealing with "large center loops" requires more elaborate tools. When extending by an elementary abelian group $A\cong C_p^d$, we will distinguish between three cases; the most favorable case occurs when $A$ is the full center of the extension loop.

\begin{definition}\label{def:Q-exact}
Let $F$ be a loop, $d$ a positive integer, and $p,r$ distinct primes. Let $L$ be a loop.
\begin{enumerate}[(a)]
\item We say that $L$ is \textit{$p^d$-exact}, if $|Z_p(L)|=p^d$, and $L$ is \textit{$p^d$-nonexact}, if $|Z_p(L)|>p^d$. 
\item The $p^d$-exact extension is \textit{unstretched,} if $|Z(L)|=p^d$. If $|Z_p(L)|=p^d$, and $|Z_r(L)|>1$, then we say that $L$ is \textit{stretched towards $r$}.
\end{enumerate}

\end{definition}

When the choice of $p^d$ is clear from the context, we will simply refer to loops as exact or nonexact. Notice that isomorphic nonexact loops may be non-isomorphic as central extensions of $F$ by a central subloop $A\cong C_p^d$. 

We define
\begin{align} \label{eq:Qbold-exact}
\boldQexact(F,p^{d}) &= \bigl\{[L]\mid |Z_{p}(L)|=p^d,\ L/Z_{p}(L) \cong F\bigr\}
\end{align}
and
\begin{align} \label{eq:Qbold-stretched}
\boldQexact(F,p^{d};\text{$r$-stretched}) &= \bigl\{[L] \in \boldQexact(F,p^{d}) \mid Z_r(L)\neq \{1\}\}.
\end{align}
One embeds $\boldQexact(F,p^{d})$ into $\boldQ(F, p^{d})$ by $[L]\mapsto [L,Z_p(L)]$. By some abuse of language, this allows us to consider $\boldQexact(F,p^{d})$ as a subset of $\boldQ(F,p^{d})$. If $A$ is a proper subloop of $Z_p(L)$, then we say that $[L, A]$ (or its representative $(L, A)$) is an \textit{nonexact extension} of $F=L/A$. In notation, 
\begin{align} \label{eq:Qbold-nonexact}
\boldQnonexact(F,p^d) &= \bigl\{[L,A]\mid A \lneq Z_p(L),\ L/A \cong F\bigr\} = \boldQ(F, p^d) \setminus \boldQexact(F, p^d).
\end{align}

The following lemma provides equivalent formulations of the fact that the central extension $L$ is nonexact or stretched. Recall that $F=L/A$, that is, an element $x\in F$ is in fact a coset $x=uA$ with $u\in L$. 

\begin{lemma}\label{lem:stretched}
Let $p$ and $r$ be distinct primes. Fix $[L,A]\in \boldQ(F,p^d)$ with $F=L/A$. Then $L$ is $r$-stretched if and only if there exists an element $z\in Z_r(F)$ such that $z$ is contained (as a coset) in $Z(L)$.   
\end{lemma}
\begin{proof}
    The loop $L$ is $r$-stretched, then there exists a nontrivial element  $u \in Z_r(L)$. Since $A$ is $p$-group with $p\neq r$, we have $u \notin A$, and hence $z=uA \in Z_r(F)$ with $z\in Z(L)$.
    Conversely, if there exists $z \in Z_r(F)$ such that $z= uA\in Z(L)$ for some $u\in L$, then $u^r \in A$ and, as $A$ has non trivial $r$-elements, $u^r=1$, so $u\in Z_r(L)\setminus \{1\}$ and $L$ is $r$-stretched.
\end{proof}

\begin{lemma}\label{lem:nonexact}
Let an extension $[L, A]\in \boldQ(F,p^d)$ be fixed with $F=L/A$. Then the following statements are equivalent.
\begin{enumerate}[(i)]
\item $[L,A]$ is an nonexact central extension.
\item $A$ is a proper subgroup of $Z_{p}(L)$, that is, $A\subsetneq Z_{p}(L)$.
\item $Z_{p}(L)/A$ is a nontrivial subgroup of $Z_{p}(F)$.
\item There exists an element $z\in Z_{p}(F)$ with $z\neq 1$ which is contained (as a coset) in $Z_{p}(L)$.
\end{enumerate}
\end{lemma}
\begin{proof}
\noindent\textit{(i)$\Longleftrightarrow$(ii)} By definition, the extension $[L,A]$ is exact if and only if $A=Z_p(L)$. Consequently, $[L,A]$ is nonexact if and only if $A\subsetneq Z_{p}(L)$.

\medskip
\noindent\textit{(ii)$\Longrightarrow$(iii)}  
Assume that $A\subsetneq Z_{p}(L)$. Let $u \in Z_p(L)\setminus A$. The element $z = uA \in F$ is a nontrivial central $p$-element of $L$ (distinct from the identity element). Consequently, $z \in Z_p(F)$, and thus $Z_p(L)/A$ is a proper subgroup of $Z_p(F)$.

\medskip
\noindent\textit{(iii)$\Longrightarrow$(ii)}  
Assume that $Z_p(L)/A$ is a proper subgroup of $Z_p(F)$. Then there exists an element $uA \in Z_p(L)/A$ such that $uA \neq A$. Since $u \in Z_p(L)$ and $uA \neq A$, it follows that $u \notin A$. Hence, $A \subsetneq Z_p(L)$.

\medskip
\noindent\textit{(iii)$\iff$(iv)}  
This equivalence follows directly from the definition.
\end{proof}

\subsection{Nilpotent loops}

The upper central series of $L$ is defined recursively by
\[
Z_{(0)}(L) = 1,\qquad L / Z_{(i+1)}(L) = Z\bigl(L / Z_{(i)}(L)\bigr),
\]
for all integers $i \ge 0$. The loop $L$ is said to be \textit{centrally nilpotent of class $n$} if
\[
Z_{(n-1)}(L) \;<\; Z_{(n)}(L) \;=\; L
\]
for some integer $n \ge 1$.
\begin{definition}
 For a positive integer $m$, we define 
 \[
 \NPL(m)=\bigl\{[L] \mid L \quad \text{is a centrally nilpotent loop of order } m \bigr\}
 \]
and 
\[
\NPL(m,p^{d})=\bigl\{[L] \in \NPL(m)\mid |Z_{p}(L)|=p^{d}\bigr\}
\] 
\end{definition}

If $p^d \nmid m$, then $\NPL(m,p^{d})=\varnothing$. If $d_1\neq d_2$, then
\[
\NPL(m,p^{d_1}) \cap \NPL(m,p^{d_2}) = \varnothing
\]

If $m$ is fixed and $p_1,p_2$ are distinct prime divisors of $m$, then the sets $\NPL(m,p_1^{d_1})$ and $\NPL(m,p_2^{d_2})$ may have a nontrivial intersection. Nevertheless, if $m$ possesses only a small number of distinct prime divisors, then the total number of nilpotent loops of order $m$ can be determined from the cardinalities $|\NPL(m,p^{d})|$ taken over all prime powers $p^d$ dividing $m$. For instance, the total number of nilpotent loops of order $20$ can be computed as
\begin{align*}
|\NPL(20)| = |\NPL(20,2)| + |\NPL(20,4)| + |\NPL(20,5)| -1,
\end{align*}
where the closing $-1$ is for the abelian group $C_2\times C_2 \times C_5$.

The enumeration of nilpotent loops of order $m$ can be reduced to the enumeration of exact central extensions of smaller loops $F$ by an elementary abelian $p$-group $A$. More precisely, as the following obvious lemma tells us, the value of $|\NPL(m,p^d)|$ can be computed from the number $|\boldQexact(F,p^d)|$ of exact extensions of smaller order nilpotent loops. 

\begin{lemma}\label{lem:nilpot-decomposition}
Assume that $p^{d}$ divides $m$. Then
\[
|\NPL(m,p^{d})|
=
\sum_{F\in\NPL(m/p^{d})}
|\boldQexact(F,p^{d})|. 
\]
\end{lemma}
\begin{proof}
The claim follows from the fact that the union
\[
\NPL(m,p^{d})
=
\mathop{\dot{\bigcup}}_{F\in\NPL(m/p^{d})}
\boldQexact(F,p^{d}),
\]
is disjoint.
\end{proof}

\section{Loop cocycles of central extensions}

Let $F$ be a finite loop with neutral element $1$, written multiplicatively, and let $A$ be a finite abelian group written additively. Consider a central extension $L$ of $F$ by $A$, and fix a system of coset representatives $\sigma(uA)$, $u \in L$, such that $\sigma(A) = 1$. Then the map
\[
(x,a) \mapsto \sigma(x)a
\]
defines a bijection $L/A \times A \to L$. Consequently, up to isomorphism of extensions, we may identify the underlying set of $L$ with the Cartesian product $F \times A$, and regard $A \leq L$ as the subset
\[
\{(1,a) \mid a \in A\}.
\]
In this framework, central extensions $L$ of $F$ by $A$ are determined by 2-cocycles $\theta: F \times F \to A$. The isomorphism classes of such pairs $[L, A]$ are then parametrized by the orbits of an appropriate affine automorphism group acting on the space of cocycles.

A function $\theta : F\times F \to A$ is called a \textit{normalized cocycle} if
\[
\theta(1,x)=\theta(x,1)=0 \qquad \text{for all } x\in F.
\]
The set of all normalized cocycles is denoted $C(F, A)$.

Let \[
\mathrm{Map}_{0}(F,A)=\{\tau:F\to A;\ \tau(1)=0\},
\]

\[
\mathrm{Hom}(F,A)=\{\tau:F\to A;\ \tau \quad \text{is a homomorphism of loops}\},
\]
 Then the mapping
\[
\partial \colon \mathrm{Map}_{0}(F,A) \to C(F,A), \qquad 
\tau \mapsto \partial{\tau}
\]
is defined by

\[
\partial{\tau}(x,y)=\tau(xy)-\tau(x)-\tau(y)
\]
is a homomorphism of groups with kernel $\mathrm{Hom}(F,A)$
The set of coboundaries forms a subgroup $B(F, A)\subseteq C(F, A)$. 
Since the kernel of the map $\tau\mapsto\partial{\tau}$ is $\mathrm{Hom}(F,A)$, we obtain
\[
B(F,A)\cong \{\tau:F\to A\mid\tau(1)=0\}/\mathrm{Hom}(F,A).
\]

The following proposition is part of the standard folklore in the field; its proof proceeds identically to that of the corresponding result for central extensions of groups.

\begin{proposition} \label{prop:cocycles}
Given $\theta\in C(F,A)$, the multiplication rule
\[
(x,a)(y,b)=\bigl(xy,\; a+b+\theta(x,y)\bigr)
\]
defines a loop on the set $F\times A$, denoted $X(F,A,\theta)$. Conversely, if $L$ is a loop and $A\leq Z(L)$, then there is a cocycle $\theta \in C(L/A,A)$ such that $L\cong X(L/A,A,\theta)$. 
\end{proposition}

Notice that in $L=X(F,A,\theta)$, $1\times A = \{(1,a) \mid a\in A\}$ is a central subloop isomorphic to $A$. We often identify $A$ with $1\times A$. 

%Moreover, if $\tau\in \mathrm{Map}_0(F,A)$, then the map $(x,a)\longmapsto (x,\; a+\tau(x))$ is an isomorphism $X(F,A,\theta)\cong X(F,A,\theta+\partial{\tau})$. 

% Since the sets $\mathrm{Map}_0(F, A)$ and $C(F, A)$ consist of $A$-valued maps, they can be considered as vector spaces over $\mathbb{F}_p$. In particular, for $x,y\in F$, the cocylce value $\theta(x,y)$ is a $d$-tuple over $\mathbb{F}_{p}$. Moreover, $\mathrm{Aut}(A)\cong GL(d,p)$, that is, the component $\beta$ of $(\alpha,\beta,\tau)\in\mathrm{Aut}^{\!*}(F,A)$
% is an invertible $d\times d$ matrix over $\mathbb{F}_p$. 

\section{The affine automorphism group and isomorphisms of extensions}

Daly and Vojt\v{e}chovsk\'y \cite{DalyVojtechovsky2009} define the group $\mathrm{Aut}(F,A)=\mathrm{Aut}(F)\times\mathrm{Aut}(A)$ and its right action 
\[
\theta^{(\alpha,\beta)}(x,y)
=\left(\theta(x^{\alpha^{-1}},y^{\alpha^{-1}})\right)^\beta
\]
on $C(F,A)$. Lemma~3.1 of \cite{DalyVojtechovsky2009} says that the map $(x,a)\longmapsto (x^{\alpha},\; a^{\beta})$ is an isomorphism $X(F,A,\theta)\cong X(F,A,\theta^{(\alpha,\beta)})$. The equivalence relation of Daly and Vojt\v{e}chovsk\'y combines two operations on cocycles, the linear action of $\mathrm{Aut}(F, A)$, and addition of coboundaries $\partial{\tau}$. In our approach, we combine these two operations into a single algebraic structure that acts affinely on both $F\times A$ and $C(F, A)$.

\begin{definition}\label{def:1}
The \textit{affine automorphism group} is
\[
\mathrm{Aut}^{\!*}(F,A)
=
\bigl\{(\alpha,\beta,\tau)
\;\big|\;
\alpha\in\mathrm{Aut}(F),\ \beta\in\mathrm{Aut}(A),\ 
\tau:F\to A,\ \tau(1)=0\bigr\},
\]
with multiplication
\[
(\alpha,\beta,\tau)(\alpha',\beta',\tau')
=
\bigl(\alpha\alpha',\; \beta\beta',\; \alpha\tau' + \tau\beta'\bigr),
\]
where
\[
(\alpha\tau')(x)=\tau'(x^{\alpha}),\qquad
(\tau\beta')(x)=(\tau(x))^{\beta'}.
\]
This multiplication arises from the composition of the block matrices.
\[
\begin{bmatrix}
\alpha & \tau \\
0 & \beta
\end{bmatrix},
\qquad
\begin{bmatrix}
\alpha' & \tau' \\
0 & \beta'
\end{bmatrix}.
\]
Thus $\mathrm{Aut}^{\!*}(F,A)$ is a group.
\end{definition}

\begin{definition}\label{def:2}
Each element $(\alpha,\beta,\tau)\in\mathrm{Aut}^{\!*}(F,A)$ acts on $F\times A$ by
\[
(x,a)^{(\alpha,\beta,\tau)}
=
\bigl(x^{\alpha},\; a^{\beta} + \tau(x)\bigr).
\]
This extends the usual action of $\mathrm{Aut}(F, A)$ by allowing translation in the $A$-coordinate.
\end{definition}

\begin{lemma}
The rule
\[
(x,a)^{(\alpha,\beta,\tau)}
=
\bigl(x^{\alpha},\,a^{\beta}+\tau(x)\bigr),
\qquad
(\alpha,\beta,\tau)\in\mathrm{Aut}^{\!*}(F,A),
\]
defines a right group action of $\mathrm{Aut}^{\!*}(F,A)$ on $F\times A$.
\end{lemma}

\begin{proof}
The identity element of $\mathrm{Aut}^{\!*}(F,A)$ is $ e=(\mathrm{id}_F,\mathrm{id}_A,0)$
where $0:F\to A$ is the zero map. For every $(x,a)\in F\times A$ we have
\[
(x,a)^e
=
\bigl(x^{\mathrm{id}_F},\,a^{\mathrm{id}_A}+0(x)\bigr)
=
(x,a)
\]
Let $g=(\alpha, \beta, \tau)$ and $g'=(\alpha', \beta', \tau')$ be the elements of $\mathrm{Aut}^{\!*}(F,A)$ and let $(x,a)\in F\times A$, then 

\[(x,a)^g
=\bigl(x^{\alpha},\,a^{\beta}+\tau(x)\bigr).
\]
and 
\begin{align*}
\bigl((x,a)^g\bigr)^{g'}
&=
\bigl((x^{\alpha})^{\alpha'},\, (a^{\beta}+\tau(x))^{\beta'}+\tau'(x^{\alpha})\bigr) \\
&=
\bigl(x^{\alpha\alpha'},\, a^{\beta\beta'}+(\tau(x))^{\beta'}+\tau'(x^{\alpha})\bigr).
\end{align*}
On the other hand, from Definition ~\ref{def:1} the product in $\mathrm{Aut}^{\!*}(F,A)$ is
\[
gg'
=
(\alpha,\beta,\tau)(\alpha',\beta',\tau')
=
(\alpha\alpha',\ \beta\beta',\ \alpha\tau'+\tau\beta'),
\]
where
\[
(\alpha\tau')(x)=\tau'(x^{\alpha}),
\qquad
(\tau\beta')(x)=(\tau(x))^{\beta'}.
\]
Therefore
\[
(x,a)^{gg'}
=\bigl(x^{\alpha\alpha'},\ a^{\beta\beta'}+(\alpha\tau'+\tau\beta')(x)\bigr)
=
\bigl(x^{\alpha\alpha'},\ a^{\beta\beta'}+(\tau(x))^{\beta'}+\tau'(x^{\alpha})\bigr).
\]
hence, 
\[
\bigl((x,a)^g\bigr)^{g'}=(x,a)^{gg'}. \qedhere
\]
\end{proof}

For $\theta \in C(F,A)$ and $(\alpha,\beta,\tau) \in \mathrm{Aut}^{\!*}(F,A)$ we define 
\begin{align} \label{eq:Astar-action}
\theta^{(\alpha,\beta,\tau)}
=
\theta^{(\alpha,\beta)} + (\partial{\tau})^{\alpha}.
\end{align}
Explicitly,
\[
\bigl(\theta^{(\alpha,\beta,\tau)}\bigr)(x,y)
=
\theta(x^{\alpha^{-1}},y^{\alpha^{-1}})^{\beta}
+
\tau((xy)^{\alpha^{-1}})
-\tau(x^{\alpha^{-1}})
-\tau(y^{\alpha^{-1}}).
\]
This extends the equivalence transformations of Daly and Vojt\v{e}chovsk\'y by combining automorphism twisting and coboundary translation. The following lemma shows the affine property of this map.

\begin{lemma} \label{lem:why-affine}
Let $\theta_1,\ldots,\theta_m\in C(F,A)$, and $k_1,\ldots,k_m$ integers with $k_1+\cdots+k_m=1$. Then
\[(k_1\theta_1+\cdots+k_m\theta_m)^{(\alpha,\beta,\tau)} = k_1\theta_1^{(\alpha,\beta,\tau)}+\cdots+k_m\theta_m^{(\alpha,\beta,\tau)}\]
\end{lemma}
\begin{proof} 
Set $\Theta=k_1\theta_1+\ldots+k_m\theta_m$. Since the map $\theta\mapsto \theta^{(\alpha,\beta)}$ is a group homomorphism of the abelian group $C(F, A)$, it is additive and commutes with integer scalar multiplication.
Hence
\[
\Theta^{(\alpha,\beta)}=
(k_1\theta_1+\ldots+k_m\theta_m)^{(\alpha,\beta)}
=(k_1\theta_1)^{(\alpha,\beta)}+\ldots+(k_m\theta_m)^{(\alpha,\beta)}.
\]
Then by formula \eqref{eq:Astar-action},
\begin{align*}
\Theta^{(\alpha,\beta,\tau)}
&=
(k_1\theta_1)^{(\alpha,\beta)}+\ldots+(k_m\theta_m)^{(\alpha,\beta)}+(\partial{\tau})^{\alpha}\\
&=
k_1\theta_1^{(\alpha,\beta)}+\ldots+k_m\theta_m^{(\alpha,\beta)}+(k_1+\ldots+k_m)(\partial{\tau})^{\alpha}\\
&=
k_1\bigl(\theta_1^{(\alpha,\beta)}+(\partial\tau)^{\alpha}\bigr)+\ldots+k_m\bigl(\theta_m^{(\alpha,\beta)}+(\partial\tau)^{\alpha}\bigr)\\
&=k_1\theta_1^{(\alpha,\beta,\tau)}+\cdots+k_m\theta_m^{(\alpha,\beta,\tau)}.
\qedhere
\end{align*}
\end{proof}

\begin{lemma} \label{lem:affine-action-on-cocycles}
The assignment
\[
C(F,A)\times\mathrm{Aut}^{\!*}(F,A)\to C(F,A),\qquad
(\theta, (\alpha,\beta,\tau))\mapsto\theta^{(\alpha,\beta,\tau)},
\]
given in formula \eqref{eq:Astar-action}, defines a right group action of
$\mathrm{Aut}^{\!*}(F,A)$ on the set of cocycles $C(F,A)$.
\end{lemma}
\begin{proof}
Since $e=(\mathrm{id}_F,\mathrm{id}_A,0)$ as above, by \eqref{eq:Astar-action},
\[
\theta^{e}
=
\theta^{(\mathrm{id}_F,\mathrm{id}_A)}+\partial 0
=
\theta+0
=
\theta.
\]

Let $g=(\alpha,\beta,\tau)$ and $h=(\gamma,\delta,\sigma)$ be elements of $\mathrm{Aut}^{\!*}(F,A)$. 
By \eqref{eq:Astar-action}, 
\[
\theta^{g}
=\theta^{(\alpha,\beta)} + (\partial{\tau})^{\alpha}.
\]
That is,
\[
\bigl(\theta^{(\alpha,\beta,\tau)}\bigr)(x,y)
=
\theta(x^{\alpha^{-1}},y^{\alpha^{-1}})^{\beta}
+
\tau((xy)^{\alpha^{-1}})
-\tau(x^{\alpha^{-1}})
-\tau(y^{\alpha^{-1}}).
\]

\begin{align*}
(\theta^{g})^{h}
&=
\bigl(\theta^{(\alpha,\beta,\tau)}\bigr)^{(\gamma,\delta,\sigma)}(x,y)
\\
&=
\Bigl(
 \theta(x^{\alpha^{-1}},y^{\alpha^{-1}})^{\beta}
 + \tau(x^{\alpha^{-1}} y^{\alpha^{-1}})
 - \tau(x^{\alpha^{-1}})
 - \tau(y^{\alpha^{-1}})
\Bigr)^{\delta}
\bigl(x^{\gamma^{-1}},y^{\gamma^{-1}}\bigr)
\\
&\qquad
+\,\sigma\bigl((xy)^{\gamma^{-1}}\bigr)
- \sigma\bigl(x^{\gamma^{-1}}\bigr)
- \sigma\bigl(y^{\gamma^{-1}}\bigr)
\\
&=
\theta\bigl(x^{(\alpha\gamma)^{-1}},y^{(\alpha\gamma)^{-1}}\bigr)^{\beta\delta}
\;+\;
\tau\bigl((xy)^{(\alpha\gamma)^{-1}}\bigr)^{\delta}
- \tau\bigl(x^{(\alpha\gamma)^{-1}}\bigr)^{\delta}
- \tau\bigl(y^{(\alpha\gamma)^{-1}}\bigr)^{\delta}
\\
&\qquad
+\,\sigma\bigl((xy)^{\gamma^{-1}}\bigr)
- \sigma\bigl(x^{\gamma^{-1}}\bigr)
- \sigma\bigl(y^{\gamma^{-1}}\bigr)
\\
&=
\theta\bigl(x^{(\alpha\gamma)^{-1}},y^{(\alpha\gamma)^{-1}}\bigr)^{\beta\delta}
\;+\;
(\tau\delta)\bigl((xy)^{(\alpha\gamma)^{-1}}\bigr)
- (\tau\delta)\bigl(x^{(\alpha\gamma)^{-1}}\bigr)
- (\tau\delta)\bigl(y^{(\alpha\gamma)^{-1}}\bigr)
\\
&\qquad
+\,(\alpha\sigma)\bigl((xy)^{(\alpha\gamma)^{-1}}\bigr)
- (\alpha\sigma)\bigl(x^{(\alpha\gamma)^{-1}}\bigr)
- (\alpha\sigma)\bigl(y^{(\alpha\gamma)^{-1}}\bigr)
\\
&=
\theta\bigl(x^{(\alpha\gamma)^{-1}},y^{(\alpha\gamma)^{-1}}\bigr)^{\beta\delta}
\;+\;
(\tau\delta + \alpha\sigma)\bigl((xy)^{(\alpha\gamma)^{-1}}\bigr)
\\
&\qquad
-\,
(\tau\delta + \alpha\sigma)\bigl(x^{(\alpha\gamma)^{-1}}\bigr)
-\,
(\tau\delta + \alpha\sigma)\bigl(y^{(\alpha\gamma)^{-1}}\bigr)
\\
&=
\bigl(\theta^{(\alpha\gamma,\beta\delta,\;\alpha\sigma+\tau\delta)}\bigr)(x,y)
\\
&=
\theta^{(\alpha\gamma,\beta\delta)}(x,y)
\;+\;
\partial(\alpha\sigma+\tau\delta)(x,y)
\\
&=
\theta^{gh}(x,y). \qedhere
\end{align*}
\end{proof}

Now, we present the affine version of \cite[Lemma 3.1]{DalyVojtechovsky2009} by showing that $\mathrm{Aut}^{\!*}(F,A)$-equivalent cocycles correspond to isomorphic loops. 

\begin{lemma} \label{lemma:1}
Let $(\alpha,\beta,\tau)\in\mathrm{Aut}^{\!*}(F,A)$. 
Then the map given by Definition~\ref{def:2} defines an isomorphism
\[
X(F,A,\theta)
\cong
X(F,A,\theta^{(\alpha,\beta,\tau)}).
\]
\end{lemma}
\begin{proof}
Let $\cdot$ and $\ast$ denotes the multiplication in $X(F,A,\theta)$ and $X(F,A,\theta^{(\alpha,\beta,\tau)})$ respectively. Then for $(x,a),(y,b)\in F\times A$, we compute,
\begin{align*}
f\big((x,a)\cdot(y,b)\big)
&= f(xy,\; a+b+\theta(x,y)) \\
&= ((xy)^\alpha,\; \tau(xy)+ a^{\beta} + b^{\beta} +\theta(x,y)^{\beta}).
\end{align*}

Using the definition of the affine action on cocycles,
$(\alpha,\beta,\tau)$ on cocycles,
\[
(\theta^{(\alpha,\beta,\tau)}(x^{\alpha},y^{\alpha}))
= \theta(x,y)^{\beta}+\tau(xy)-\tau(x)-\tau(y).
\]
Hence
\[
(\theta(x,y)^{\beta}) =\theta^{(\alpha,\beta,\tau)}( x^{\alpha},y^{\alpha})
 -\tau(xy)+\tau(x)+\tau(y).
\]

Substituting into the previous expression gives
\begin{align*}
f\big((x,a)\cdot(y,b)\big)
&= ((x^{\alpha})(y^{\alpha}),\; \tau(xy)\!+\! a^{\beta}\!+\! b^{\beta}
 \!+\! \theta^{(\alpha,\beta,\tau)}(x^{\alpha},y^{\alpha}) 
 \!-\!\tau(xy)\!+\!\tau(x)\!+\!\tau(y))\\
&= ((x^{\alpha})(y^{\alpha}),\; \tau(x)+ a^{\beta}+\tau(y)+ b^{\beta}
 +\theta^{(\alpha,\beta,\tau)}(x^{\alpha},y^{\alpha})) \\
&= (x^{\alpha},\tau(x)+ a^{\beta}) \ast(y^{\alpha},\tau(y)+b^{\beta}) \\
&= f(x,a)\ast f(y,b).
\end{align*}

Thus $f$ is a homomorphism. Its inverse is
\[
(x,a)\mapsto(x^{\alpha^{-1}}, a^{\beta^{-1}} -\tau(x)),
\]
Thus, $f$ is bijective and hence an isomorphism.
\end{proof}

In fact, the converse also holds.

\begin{lemma} \label{lemma:2}
Conversely, let 
$f:X(F,A,\theta)\to X(F,A,\mu)$ 
be a loop isomorphism satisfying $f(A)=A$. 
Then $f$ is induced by a triple 
$(\alpha,\beta,\tau)\in\mathrm{Aut}^{\!*}(F,A)$, and
\[
\mu=\theta^{(\alpha,\beta,\tau)}.
\]
\end{lemma}

\begin{proof}
The underlying set of both loop extensions is $F\times A$. 
Since $f$ is a homomorphism, we have
\[
f:F\times A\longrightarrow F\times A,
\qquad
(x,a)\longmapsto\bigl(u(x,a),\,v(x,a)\bigr).
\]
Because $f(A)=A$, for every $a\in A$ we must have
\[
f(1,a)=(1,a')=\big(u(1,a),v(1,a)\big) \quad\text{for a unique }a'\in A.
\]
Define
\[
a^\beta=v(1,a).
\]
Then
\[
f(1,a)=(1,a^{\beta}).
\]
Using that $f$ is a homomorphism,
\[
f\bigl((1,a)(1,b)\bigr)
=
f(1,a+b)
=
(1,(a+b)^{\beta}),
\]

but also
\[
f(1,a)*f(1,b)=(1,a^{\beta})*(1,b^{\beta})
=(1,a^{\beta}+b^{\beta}).
\]
Hence
\[
(a+b)^{\beta}=a^{\beta}+b^{\beta},
\]
so $\beta$ is an automorphism of $A$.\\
For each $x\in F$ we may write
\[
f(x,0)=(x',c)=(u(x,0),\,v(x,0))
\quad\text{with }x'\in F,\ c\in A.
\]
Define
\[
x^\alpha=u(x,0),\qquad \tau(x)=v(x,0).
\]
Since $f$ preserves the identity element,
\[
f(1,0)=(1,0),
\]
we obtain
\[
1^\alpha=1,\qquad \tau(1)=0.
\]

Now use the homomorphism property on $(x,0)\cdot(y,0)=(xy,\theta(x,y))$.
\begin{align*}
f(xy,\theta(x,y)) &= f((x,0)\cdot(y,0))\\
&= f(x,0)*f(y,0)\\
&= (x^{\alpha},\tau(x))*(y^{\alpha},\tau(y))\\
&=\bigl(x^{\alpha}y^{\alpha},\tau(x)+\tau(y)+\mu(x^{\alpha},y^{\alpha})\bigr).
\end{align*}

also
\begin{align*}
f(xy,\theta(x,y)) &= \big((xy)^{\alpha},\tau(xy)+\theta(x,y)^{\beta}\big) 
\end{align*}

Comparing the first coordinates gives
\[
(xy)^{\alpha}=x^{\alpha}y^{\alpha},
\]
so $\alpha\in\mathrm{Aut}(F)$.

Comparing the second coordinates gives
\[
\mu(x^{\alpha},y^{\alpha})
=
\theta(x,y)^{\beta}
+\tau(xy)-\tau(x)-\tau(y),
\]
that is,
\[
\mu=\theta^{(\alpha,\beta,\tau)}.
\]
In every loop extension one has $(x,a)=(x,0)(1,a)$, apply $f$ on both sides
\begin{align*}
f(x,a)&=f(x,0)*f(1,a)\\
&=(x^{\alpha},\tau(x))*(1,a^{\beta})\\
&=(x^{\alpha},\tau(x)+a^{\beta}+\mu(x^{\alpha},1)).
\end{align*}

Since $\mu(x^{\alpha},1)=0$. Therefore,
\[
f(x,a)=\bigl(x^{\alpha},\,\tau(x)+a^{\beta}\bigr).
\]
Thus the values of $f(x,0)$ and $f(1,a)$ determine $\alpha$, $\beta$, and $\tau$ uniquely.
\end{proof}

\begin{theorem}\label{thm:Q-orbits}
For every loop $F$ and every abelian group $A$, there is a bijection between $\boldQ(F,A)$ and the orbits of $\mathrm{Aut}^{\!*}(F,A)$ on $C(F,A)$. 
\end{theorem}
\begin{proof}
By Lemma~\ref{lemma:1}, each element
$(\alpha,\beta,\tau)\in\mathrm{Aut}^{\!*}(F,A)$ induces an isomorphism $X(F,A,\theta)\cong X(F,A,\theta^{(\alpha,\beta,\tau)})$,
hence all cocycles in the same $\mathrm{Aut}^{\!*}(F,A)$-orbit determine isomorphic extensions. 
Conversely, Lemma~\ref{lemma:2} shows that every isomorphism
$X(F,A,\theta)\cong X(F,A,\mu)$ with $f(A)=A$ is induced by some $(\alpha,\beta,\tau)\in\mathrm{Aut}^{\!*}(F,A)$, and then $\mu=\theta^{(\alpha,\beta,\tau)}$. 
Thus the isomorphism classes $[L,A]\in \boldQ(F,A)$ are in bijection with the orbits of $\mathrm{Aut}^{\!*}(F,A)$ on $C(F,A)$.
\end{proof}

\section{Orbit counting by Burnside's lemma}\label{sec:burnside}

Theorem~\ref{thm:Q-orbits} shows that for finite $F$ and $A$, $|\boldQ(F,A)|$ is equal to the number of orbits of $\mathrm{Aut}^{\!*}(F,A)$ on $C(F,A)$. We can compute the number of these orbits using Burnside's lemma \cite{rotman2012introduction}, which expresses the number of orbits as a sum of fixed point counts. For a finite group $G$ acting on a finite set $X$, one has
\[
 |X/G| \;=\; \frac{1}{|G|} \sum_{g\in G} | \mathrm{Fix}(g)|
\]
where $\mathrm{Fix}(g)=\{x\in X \mid x^{g}=x\}$ is the fixed point set of $g$. The function $g\mapsto |\mathrm{Fix}(g)|$ is constant on every conjugacy class of $G$. Hence, Burnside's lemma reduces to
\[
|C(F,A)\,/\,G|
=
\frac{1}{|G|}
\sum_{C\in\mathrm{Conj}(G)}
|C|\cdot |\mathrm{Fix}(g_C)|,
\]
where $g_{C}$ is any representative of the conjugacy class $C$. Summing over the conjugacy classes yields the number of distinct orbits.

From a practical point of view, for small $F$ and $A$, the group $G=\mathrm{Aut}^{\!*}(F, A)$ and its conjugacy classes are easily computable. However, the cocycle space $C(F, A)$ can be huge even for small $F$ and $A$, and the computation of $|\mathrm{Fix}(g)|$ can be challenging. If $A=C_p^d$ is elementary abelian, then we can control the situation as follows. Since $A$ is a vector space over $\mathbb{F}_p$, so does $C(F,A)$. By Lemma \ref{lem:why-affine}, the action of $G$ is affine over $\mathbb{F}_p$, and for $g\in G$, the set $\mathrm{Fix}(g)$ is an affine subspace. Linear algebra methods can compute their affine rank; see Appendix \ref{app:matrix}. If $|C(F, A)|=|A|^{(|F|-1)^2}$ is not very large, then the implementation can also compute the orbit representatives. Recently, determining orbit representatives is feasible for $|C(F, A)|\leq 10^6$ on a personal computer using GAP \cite{GAP4} and the LOOPS package \cite{nagy2007loops}.

\section{Computing exact extensions}\label{sec:exact-extensions}

In this section, we present a useful method to speed up the calculation of the number $|\boldQexact(F, A)|$ of exact extensions. In the whole section, $A=C_p^d$ is an elementary abelian group of order $p^d$. In the notation $\boldQ(F,A)$, $\boldQexact(F,A)$, $\mathrm{Aut}^{\!*}(F,A)$, $C(F,A)$, we will replace $A$ by the prime power $p^d$: $\boldQ(F,p^d)$, $\boldQexact(F,p^d)$, $\mathrm{Aut}^{\!*}(F,p^d)$, $C(F,p^d)$. Any central extension $L$ of $F$ by $A$ has underlying set $F\times A$. The loop $L$ is represented by a cocycle $\theta \in C(F,p^d)$; $L=X(F,p^d,\theta)$. The surjective homomorphism $\phi:L\to F\cong L/A$ is given by $(x,a) \mapsto x$. For $x\in F$, we write $\bar{x}=(x,0)$. The full preimage of $x$ is $\phi^{-1}(x)=\{(x,a) \mid a\in A\}=\bar x A$. 

Lemma~\ref{lem:nonexact} shows that $L$ is nonexact if and only if there exists $z\in Z_{p}(F)$ such that the full preimage $\phi^{-1}(z)$ is contained in $Z_{p}(L)$. Similarly, $L$ is stretched towards the prime $r\neq p$ in and only if there is a $z\in Z_r(F)$ such that $\phi^{-1}(z)$ is contained in $Z(L)$. 

\begin{lemma}
Let $L=X(F,p^d,\theta)$.
\begin{enumerate}[(i)]
\item The element $(z,a)$ is in $Z(L)$ if and only $z\in Z(F)$ and
\begin{equation}\label{eq:theta-central}
\begin{aligned}
 \theta(x,z) &= \theta(z,x),\\
 \theta(x,y) + \theta(xy,z) &= \theta(y,z) + \theta(x,yz),\\
 \theta(x,z) + \theta(xz,y) &= \theta(z,y) + \theta(x,zy),\\
 \theta(z,x) + \theta(zx,y) &= \theta(x,y) + \theta(z,xy)\\
\end{aligned}
\end{equation}
hold for all $x,y\in F$. 
\item The element $(z,a)$ is in $Z_p(L)$ if and only $z\in Z_p(F)$ and
\begin{equation}\label{eq:theta-order-p}
 \theta(z,z) + \theta(z^{2},z) + \cdots + \theta(z^{p-1},z) = 0.
\end{equation}
\end{enumerate}
\end{lemma}
\begin{proof}
(i) is given in \cite[Section 8]{DalyVojtechovsky2009}. (ii) follows from the general formula
\[
\bar z^k=(z^k, \theta(z,z)+\theta(z^2,z)+\cdots+\theta(z^{k-1},z)),
\]
which can be shown by induction on $k$. 
\end{proof}

\begin{definition}\label{def:nonexact-cocycle}
Let $F$ be a loop, $d$ a positive integer, $p,r$ primes and $z\in Z_r(F)$. Define the subspace
\begin{align*}
C(F,p^{d},z) =
\begin{cases}
\bigl\{\theta\in C(F,p^{d})\mid \text{\eqref{eq:theta-central} and \eqref{eq:theta-order-p} hold for all $x,y\in F$} \bigr\} & \text{if $p=r$}, \\
\bigl\{\theta\in C(F,p^{d})\mid \text{\eqref{eq:theta-central} hold for all $x,y\in F$} \bigr\} & \text{if $p\neq r$}
\end{cases}
\end{align*}
of cocycles. 
\end{definition}

Notice that if $z^i\neq 1$, then $C(F,p^{d},z)=C(F,p^{d},z^i)$. In fact, $C(F,p^{d},z)$ only depends on $F$, $p^d$ and the central subgroup $\langle z \rangle$ of prime order. We also remark that the cocycle spaces $C(F,p^d,z)$ and the action of $\mathrm{Aut}^{\!*}(F,p^{d})$ on these subspaces are computationally well tractable.

\begin{lemma} \label{lem:restricted-ext}
For every $(\alpha,\beta,\tau)\in\mathrm{Aut}^{\!*}(F,p^{d})$
\begin{enumerate}[(i)]
 \item $C(F,p^{d},z)$ is an $\mathbb{F}_{p}$-linear subspace of $C(F,p^{d})$, containing all coboundaries $\partial \tau(x,y)$.
 \item $C(F,p^{d},z)^{(\alpha,\beta,\tau)}=C(F,p^{d},z^{\alpha})$.
 \item The group $\mathrm{Aut}^{\!*}(F,p^{d})$ acts on $\bigcup\limits_{z \in Z_{p}(F)} C(F, p^{d}, z)$.
 \item The group
 \[\mathrm{Aut}^{\!*}(F,p^d,z)=\{(\alpha,\beta,\tau) \in \mathrm{Aut}^{\!*}(F,p^d) \mid \langle z \rangle=\langle z^\alpha \rangle\}\]
 has a natural affine action on $C(F,p^d,z)$. 
\end{enumerate}
\end{lemma}
\begin{proof}
Fix $z \in Z_p{(F)}\setminus \{1\}$.

\noindent\textit{(i)}
The conditions \eqref{eq:theta-central} and \eqref{eq:theta-order-p} are homogenous linear equations in values of $\theta$, hence, $C(F,p^d,z)$ is an $\mathbb{F}_p$-linear subspace of $C(F,p^d)$. If $\theta=\partial\tau$ is a coboundary, then $X(F,p^d,\theta)$ is isomorphic to the extension $F\times C_p^d$, in which $(z,0)$ is the central element of order $p$. By previous lemma $\partial\tau$ satisfies \eqref{eq:theta-central} and \eqref{eq:theta-order-p} so $\partial\tau \in C(F,p^d,z)$.

\noindent\textit{(ii)}
Let $(\alpha,\beta,\tau) \in \mathrm{Aut}^{\!*}(F,p^d)$ and $\theta \in C(F,p^d,z)$, then by Lemma~\ref{lemma:1} $X(F,p^d,\theta) \cong X(F,p^d,\theta^{(\alpha,\beta,\tau)})$. Since the element $(z,0)$ is central element of order $p$ in $X(F,p^d,\theta)$ so by isomorphism $(z^{\alpha},\tau(z))$ is central element of order $p$ in $X(F,p^d,\theta^{(\alpha,\beta,\tau)})$. By \eqref{eq:theta-central} and \eqref{eq:theta-order-p} this is equivalent to $\theta^{(\alpha,\beta,\tau)} \in C(F,p^d,z^\alpha)$. Therefore, $C(F,p^d,z)^{(\alpha,\beta,\tau)}=C(F,p^d,z^\alpha)$.

\noindent\textit{(iii)}
Since $\alpha \in \mathrm{Aut}(F)$ implies $\alpha(Z_p(F))=Z_p(F)$. Therefore from $(ii)$, $\mathrm{Aut}^{\!*}(F,p^d)$  acts on $\bigcup\limits_{z \in Z_{p}(F)} C(F, p^{d}, z)$.

\noindent\textit{(iv)}
By definition
\[\mathrm{Aut}^{\!*}(F,p^d,z)=\{(\alpha,\beta,\tau) \in \mathrm{Aut}^{\!*}(F,p^d) \mid \langle z \rangle=\langle z^\alpha \rangle\}\]
if $(\alpha,\beta,\tau) \in \mathrm{Aut}^{\!*}(F,p^d,z)$, then $z^\alpha \in \langle z \rangle$ and $C(F,p^d,z^\alpha)=C(F,p^d,z)$. 
Now by $(ii)$, $C(F,p^d,z)$ is invariant under $\mathrm{Aut}^{\!*}(F,p^d,z)$. Therefore, the affine action on $C(F,p^d)$ restricts to an affine action on $C(F,p^d,z)$.
\end{proof}

Theorem~\ref{thm:Q-orbits} leads us to introduce the following notation. We write $\boldO(F,p^d,z)$ for the set of $\mathrm{Aut}^{\!*}(F,p^d,z)$-orbits in $C(F,p^d,z)$. Thus, determining $|\boldO(F,p^{d},z)|$ amounts to counting orbits in the smaller space $C(F,p^{d},z)$. One must be careful, though, because the notion of isomorphism in $\boldQ(F,p^d)$ and equivalence in $\boldO(F,p^d,z)$ is not identical.

\begin{definition}\label{def:ptame}
Let $F$ be a finite nilpotent loop and let $p,r$ be distinct primes.
We say that $F$ is \textit{$p$-tame} if for every $d\ge 1$ and for every $z\in Z_{p}(F)\setminus\{1\}$, one has
\begin{align} \label{eq:p-tame}
|\boldQnonexact(F,p^{d})|=|\boldO(F,p^{d},z)|,
\end{align}
and 
\begin{align} \label{eq:p-r-tame}
|\boldQexact(F,r^{d}; \text{$p$-stretched})|=|\boldO(F,r^{d},z)|.
\end{align}
\end{definition}

\begin{theorem} \label{thm:p-tame-conditions}
\begin{enumerate}[(i)]
    \item If $|Z_{p}(F)|=p$, then $F$ is $p$-tame.
    \item If $F$ is an elementary abelian $p$-group, then $F$ is $p$-tame.
    \item If $|F|=p^2$, then $F$ is $p$-tame.
\end{enumerate}
\end{theorem}
\begin{proof}
We establish that the given assumptions entail \eqref{eq:p-tame} for every $z \in Z_p(F)\setminus \{1\}$. An analogous argument yields the corresponding claim for \eqref{eq:p-r-tame}.

\noindent\textit{(i)}
Let $Z_{p}(F)=\langle z\rangle$. Then $\langle z^{\alpha} \rangle = \langle z\rangle$ for every $\alpha\in\mathrm{Aut}(F)$. Consequently,
\[
\mathrm{Aut}^{\!*}(F,p^{d},z)=\mathrm{Aut}^{\!*}(F,p^{d}).
\]
By Lemma~\ref{lem:nonexact}, the nonexact extensions correspond exactly to the $\mathrm{Aut}^{\!*}(F,p^{d})$ orbits on $C(F,p^{d},z)$, that is, $\boldQnonexact(F,p^d) = \boldO(F,p^{d},z)$. 

\noindent\textit{(ii)} Let $F\cong C_{p}^{n}$ with $n\ge 1$, so $Z_{p}(F)=F$.
The group $\mathrm{Aut}(F)\cong \mathrm{GL}(n,p)$ acts transitively on
\[
\bigl\{\langle z\rangle \mid z\in Z_{p}(F),\ z\neq 1\bigr\},
\]
Fix $z \in Z_p(F) \setminus \{1\}$. Let $\theta\in C(F,p^{d})$, and assume that the corresponding extension $L=X(F,p^{d},\theta)$ is nonexact.
By Lemma~\ref{lem:nonexact}, there exists $z'\in Z_{p}(F)$, $z'\neq 1$, such that
\[
(z',0)\in Z_{p}(L).
\]
Choose $\alpha\in\mathrm{Aut}(F)$ such that $z'^{\alpha}=z$.
By Lemma~\ref{lemma:1}, the cocycle $\theta^{(\alpha,\mathrm{id},0)}$
defines an isomorphic extension and satisfies
\[
(z,0)\in Z_{p}\bigl(X(F,p^{d},\theta^{(\alpha,\mathrm{id},0)})\bigr),
\]
that is,
\[
\theta^{(\alpha,\mathrm{id},0)}\in C(F,p^{d},z).
\]
Hence every nonexact $\mathrm{Aut}^{\!*}(F,p^{d})$ orbit on $C(F,p^{d})$
contains an element of $C(F,p^{d},z)$.

Now let $\theta_{1},\theta_{2}\in C(F,p^{d},z)$ and assume that
\[
\theta_{2}=\theta_{1}^{(\alpha,\beta,\tau)}
\]
for some $(\alpha,\beta,\tau)\in\mathrm{Aut}^{\!*}(F,p^{d})$.
If $\langle z\rangle^{\alpha}=\langle z \rangle$, then $(\alpha,\beta,\tau)\in\mathrm{Aut}^{\!*}(F,p^{d},z)$, 
and the claim follows.
Assume therefore that $\langle z\rangle^{\alpha}=\langle z' \rangle \neq \langle z \rangle$.

Set $L_{2}=X(F,p^{d},\theta_{2})$ and write $\bar z=(z,0)$, $\bar{z}'=(z',0)$.
Since $\theta_{2}\in C(F,p^{d},z)$, we have $\bar z\in Z_{p}(L_{2})$.
Moreover, since $\theta_{2}\in C(F,p^{d},z')$, we also have
$\bar z'\in Z_{p}(L_{2})$.
Therefore $Z_{p}(L_{2})/A$ contains the subgroup
$\langle z,z'\rangle\le F$ of order $p^2$. Since $F$ is elementary abelian, there exists a complement
$\bar D\le F$ such that
\[
F=\langle z,z'\rangle \oplus \bar D.
\]
Let $D=\bar D \times A \le L_{2}$ be the full preimage of $\bar D$ under the natural projection $L_{2}\to L_{2}/A\cong F$. Define the $p$-central subgroup $B=\langle (z,0), (z',0) \rangle$ of $L_2$. Then $D\cap B=0$ and $L_{2}=DB$, hence $L_2$ is the direct product of the loop $D$ with the elementary abelian group $B$. Therefore, the automorphism group of $L_{2}$ acts transitively on $B\setminus \{1\}$, that is, there exists $\varphi\in\mathrm{Aut}(L_{2})$ such that
$\bar z^{\varphi}=\bar z'$, and $\varphi|_D=\mathrm{id}_D$.
Since $A\leq D$, $\varphi(A)=A$. By Lemma~\ref{lemma:2}, the automorphism $\varphi$ is induced by a triple 
$(\alpha_{0},\beta_{0},\tau_{0})\in\mathrm{Aut}^{\!*}(F,p^{d})$ satisfying
\[
\theta_{2}^{(\alpha_{0},\beta_{0},\tau_{0})}=\theta_{2}
\quad\text{and}\quad
z^{\alpha_{0}}=z'.
\]
Consequently,
\[
\theta_{2}
=\theta_{1}^{(\alpha,\beta,\tau)}
=\theta_{1}^{(\alpha,\beta,\tau)(\alpha_{0},\beta_{0},\tau_{0})^{-1}},
\]
and the $\mathrm{Aut}(F)$ component of
$(\alpha,\beta,\tau)(\alpha_{0},\beta_{0},\tau_{0})^{-1}$
fixes $\langle z \rangle$.
Thus $\theta_{1}$ and $\theta_{2}$ lie in the same
$\mathrm{Aut}^{\!*}(F,p^{d},z)$ orbit.

\noindent\textit{(iii)} If $F$ is nilpotent of order $p^2$, then either $|Z_p(F)|=p$ or $F$ is elementary abelian. The result follows from $(i)$ and $(ii)$. 
\end{proof}

\begin{remark}
All nilpotent loops of order less than 12 are 2-tame, except for the abelian group $C_2\times C_4$. Indeed, if $F$ is a nilpotent loop of order less than 12, and $F$ is not 2-tame, then $|Z_2(F)|\geq 4$, and $F$ is an abelian group, which is not elementary abelian. 
\end{remark}

\begin{lemma} \label{lem:C4C2}
Let $F=C_4\times C_2$, $z_1,z_2\in F$ such that $F/\langle z_1 \rangle \cong C_2\times C_2$ and $F/\langle z_2 \rangle \cong C_4$. Then for all prime $p$ and positive integer $d$,
\begin{align} \label{eq:C4C2-ext}
|\boldQnonexact(F,2^d)|=|\boldO(F,2^d,z_1)|+|\boldO(F,2^d,z_2)|-1.
\end{align}
and
\begin{align} \label{eq:C4C2-stretched}
|\boldQexact(F,p; \text{$2$-stretched})|=|\boldO(F,p,z_1)|+|\boldO(F,p,z_2)|-1.
\end{align}
\end{lemma}
\begin{proof}
We first notice that $F$ has three involutions $z_1,z_2,z_1z_2$, and $F/\langle z_1z_2 \rangle \cong C_4$. Moreover, there is an $\alpha\in\mathrm{Aut}(F)$ with $z_2^\alpha=z_1z_2$. By Lemma \ref{lem:restricted-ext}, cocycles in $C(F,2,z_2)$ and $C(F,2,z_1z_2)$ define isomorphic extensions. All nonexact extensions are in $\boldO(F,p,z_1)$ or $\boldO(F,p,z_2)$. The only loop contained in both is the abelian group $C_4\times C_2\times C_p$. 
\end{proof}

\section{Experimental results on nilpotent loops of order less than 24}

Using our orbit-counting techniques, we sought to reproduce the computational findings of \cite{DalyVojtechovsky2009}. More specifically, we not only verified the figures reported by Daly and Vojt\v{e}chovsk\'y, but also used them as a benchmark to validate our methodology. In this section, we describe our computational procedure in greater detail and present the resulting data. Our starting point is that for a $p$-tame nilpotent loop $F$, formulas \eqref{eq:p-tame} and \eqref{eq:p-r-tame} enable us to efficiently compute $|\boldQexact(F,p^{d})|$ and $|\boldQexact(F,p^{d};\text{$r$-stretched})|$ by orbit-counting. Hence, we are able to compute the values
\begin{align} \label{eq:NL-m-pd}
TS(m,p^d) = \sum_{\substack{F\in\NPL(m) \\ \text{$F$ is $p$-tame}}} |\boldQexact(F,p^{d})|
\end{align}
and
\begin{align} \label{eq:NL-m-pd-r}
TS(m,p^d;r) = \sum_{\substack{F\in\NPL(m) \\ \text{$F$ is $p$-tame}}} |\boldQexact(F,p^{d})|-|\boldQexact(F,p^d;\text{$r$-stretched})|
\end{align}
for $m<12$. We use Lemma \ref{lem:C4C2} to deal with the exceptional case $F=C_4\times C_2$. 

Let us emphasize that a loop in $\boldQexact(F,p^{d})$ may have a center which is larger than $p^d$. Also, for $p_1\neq p_2$, $\boldQexact(F_1,p_1^{d_1})$ and $\boldQexact(F_2,p_2^{d_2})$ may have nonempty intersection. For example, $C_3\times C_2$ is contained in both $\boldQexact(C_3,2)$ and $\boldQexact(C_2,3)$. Also, $TS(6,2)$ counts the abelian group $C_3\times C_4$, $TS(3,4)=|\boldQ(C_3,4)|$ counts $C_3\times C_2^2$, and $TS(4,3)$ counts both of these groups. 

% The following proposition shows that if $F$ is small, then we can compute the number of stretched extensions by orbit-counting.

% \begin{proposition}
% Let $p$ and $r$ be distinct primes, and $F$ a nilpotent loop such that $p\nmid |F|$ and $|Z_r(F)|=r$. Let $w\in Z_r(F)\setminus \{1\}$, and define $C'(F,p^d,w)\leq C(F,p^d)$ to be the space of cocycles satisfying equation \eqref{eq:theta-central} for $w$. Then $\mathrm{Aut}^{\!*}(F,p^d)$ acts on $C'(F,p^d,w)$. Moreover, 
% \begin{align*} %\label{eq}
% \theta \mapsto [X(F,p^d,\theta)]
% \end{align*}
% induces a bijection between the orbits and the set $\boldQexact(F,p^d;\text{$r$-stretched})$ of stretched extensions. 
% \end{proposition}
% \begin{proof}
% Since $|Z_r(F)|=r$ and $\langle w \rangle = Z_r(F)$, one has $\langle z \rangle^\alpha =\langle z \rangle$ for all $\alpha\in \mathrm{Aut}(F)$. This implies that $\mathrm{Aut}^{\!*}(F,p^d)$ acts on $C'(F,p^d,w)$. The second claim is immediate. 
% \end{proof}

% We can compute the number of $\mathrm{Aut}^{\!*}(F,p^d)$-orbits on $C'(F,p^d,w)$ by orbit-counting. Our notation will be

\begin{table}
\caption{The number of nilpotent loops of order at most 22 \label{tab:less-than-24}}
\scriptsize
\begin{tabular}{lrrp{55mm}}
$\mathbf{n}$ & \textbf{Number of nilpotent loops}          & \textbf{time (ms)} & \textbf{Strategy}               \\
8          & 139                                         & 262                     & $|\boldQ(C_4,2)| + |\boldQ(C_2\times C_2,2)| -1$     \\
9          & 10                                          & 29                     & $|\boldQ(C_3,3)|$                       \\
10         & 1'044                                      & 26'045                 & $|\boldQ(C_5,2)|$                      \\
12         & 2'623'755                                 & 1'315                  & $TS(4,3;2)+TS(6,2)+|\boldQ(C_3,4)|-2$ \\
14         & 178'962'784                               & 186                    & $|\boldQ(C_7,2)|$                      \\
15         & 66'630                                     & 152                    & $|\boldQ(C_5,3)|+|\boldQ(C_3,5)|-1$           \\
16         & 466'409'543'467'341                     & 166'141                & $TS(2,8)+TS(4,4)+TS(8,2)+$        \\
&&                                                                            & $|\boldQexact(C_4\times C_2,2)|$ \\
18         & 157'625'998'010'363'396                & 8'792                  & $TS(6,3;2)+TS(9,2)$           \\
20         & 4'836'883'870'081'433'134'085'047    &193'840                & $TS(10,2)+TS(4,5)+|\boldQ(C_5,4)|$  \\
21         & 17'157'596'742'633                      & 546                    & $|\boldQ(C_7,3)|+|\boldQ(C_3,7)|-1$           \\
22         & 123'794'003'928'541'545'927'226'368 & 993                   & $|\boldQ(C_{11},2)|$                     
\end{tabular}
\end{table}

Our findings are summarized in Table~\ref{tab:less-than-24}. In all cases, we obtained the same values as Daly and Vojt\v{e}chovsk\'y~\cite{DalyVojtechovsky2009}. Those authors reported that enumerating nilpotent loops of order $20$ required approximately 2 days of computation on a single-processor Unix machine using the GAP \cite{GAP4} computer algebra system. By contrast, our computations \cite{nagygp2025nilpotentloops} required only seconds to minutes on a standard personal computer. 

This improvement in performance is principally due to recasting the original classification task as an orbit-counting problem arising from an affine group action. This reformulation reduces the isomorphism testing procedure to computations in linear algebra over finite fields and, in turn, facilitates the systematic application of Burnside’s lemma. In the enumeration of nilpotent loops, our focus is on the $p$-center $Z_p(L)$. The notions of exact and stretched extensions, combined with the $p$-tame property of small nilpotent loops, further reduce the problem to counting orbits. Consequently, rather than repeatedly performing isomorphism reduction over a large cohomological search space, restricting to $p$-centers via $p$-tameness confines the orbit-counting problem to significantly smaller linear subspaces.

\section{Enumeration of nilpotent loops of order 24 with large center}

\begin{table}
\caption{The number of nilpotent loops of order 24 with a center of size at least 3 \label{tab:order-24}}
\scriptsize
\begin{tabular}{p{8cm}r}
$2$-center of size $8$ & 4 \\ 
$2$-center of size $4$ & 641'384'579'065 \\ 
$3$-center of size $3$ & 2'318'640'494'101'035'419'767 \\ 
\hline
$2$-center of size $8$ and $3$-center of size 3 & 1 \\ 
$2$-center of size $4$ and $3$-center of size 3 & 1 \\ 
Center $C_2\times C_3$ & 39'800 \\ 
\hline\hline
\textbf{Total number of nilpotent loops with} & 2'318'640'494'742'419'998'834 \\ 
\textbf{a center of size at least 3} & 
\end{tabular}
\end{table}

\begin{theorem} \label{thm:large-center-24}
Table \ref{tab:order-24} lists the numbers of nilpotent loops of order 24 whose center has size at least 3, together with the counts corresponding to specified sizes of their 2- and 3-centers.
\end{theorem}
\begin{proof}
In lines 1--3 of Table \ref{tab:order-24} we have the classes $\NPL(24,8)$, $\NPL(24,4)$ and $\NPL(24,3)$. We compute their cardinalities by
\begin{align*}
|\NPL(24,8)| &= |\boldQ(C_3,8)|, \\
|\NPL(24,4)| &= TS(6,4), \\
|\NPL(24,3)| &= \sum_{F\in \NPL(8)} |\boldQ(F,3)|.
\end{align*}
Moreover,
\begin{align*}
\NPL(24,8) \cap \NPL(24,4) &= \varnothing, \\
\NPL(24,8) \cap \NPL(24,3) &= \{C_2^3 \times C_3\}, \\
\NPL(24,4) \cap \NPL(24,3) &= \{C_2\times C_4 \times C_3\}.
\end{align*}
This yields the total number of nilpotent loops of order 24 with a center of size at least 3. 
The set of nilpotent loops whose center contains $C_6$ is the disjoint union of the classes $\boldQexact(F,3;\text{$2$-stretched})$, where $F\in \NPL(8)$. A loop in this class with a strictly larger center must be one of the three abelian groups of order 24. Hence, the number of nilpotent loops of order 24 with a center of size 6 is
\begin{align*}
\sum_{F\in \NPL(8)} |\boldQexact(F,3;\text{$2$-stretched})|-3,
\end{align*}
which is computable by Theorem \ref{thm:p-tame-conditions} and Lemma \ref{lem:C4C2}. 
\end{proof}

The computation of the values in Table \ref{tab:order-24} takes about 100 minutes on a standard personal computer. 

We conclude this paper with three remarks; two of them explain why our method cannot compute the number of nilpotent loops of order 24 with a center of size 2.
\begin{enumerate}
\item We know that there are over $2.5$ million nilpotent loops $F$ of order 12, but we have no explicit list of them. The computation of $|\boldQ(F,2)|$ for all $F\in \NPL(12)$ would take years.
\item The standard GAP routines for matrix groups become very slow when the dimension is high, and the group has a large elementary abelian part. This is the case for $\mathrm{Aut}^{\!*}(F,2)$ for most of the loops $F\in \NPL(12)$. 
\item Our method does not offer a straightforward way to compute the number of nilpotent loops with a cyclic center of order 4 or 8. A possible solution would be the implementation of the linear algebra routines over the ring $\mathbb{Z}/m\mathbb{Z}$. 
\end{enumerate}

\section{Conclusion}

In this paper, we reformulate the cohomological classification and enumeration of central extensions of a nilpotent loop $F$ by an elementary abelian $p$-group $A$ by replacing the usual qoutient viewpoint $H(F,A)=C(F,A)/B(F,A)$ with an explicit orbit problem on the full cocycle space $C(F,A)$. We introduce the affine automorphism group  $\mathrm{Aut}^{\!*}(F,A)$, which combines the natural  $\mathrm{Aut}(F)$ and $\mathrm{Aut}(A)$ action with translations by coboundaries. We prove that its orbits on $C(F,A)$ are in bijection with isomorphism classes $[L,A]$ of central extensions with $L/A \cong F$. This orbit perspective enables a direct application of Burnside's lemma, where the fixed point count becomes linear algebra over finite fields. This supports further reductions via the notion of exact and stretched extensions and the restriction on cocycle subspaces $C(F,p^d,z)$ under tameness conditions. Computationally, the approach reproduces the known enumeration of nilpotent loops of order less than 24 while achieving substantial performance improvements. Our method also gives the counts of nilpotent loops of order 24 with center of size at least 3. Finally, we clarify the current limitations of the method for center size 2 and for cyclic centers of orders 4 and 8. Hence, it identifies several specific challenges that remain open for future work.

\bibliographystyle{abbrvnat}
\bibliography{references}

%%%%%%%%%%%%%%%%%%%%%%%%%%%%%%%%%%%%%%%%%%%%%

\appendix

\section{Matrix interpretation of the affine action}\label{app:matrix}

We describe the affine action of $\mathrm{Aut}^{\!*}(F,A)$ on cocycles introduced in \eqref{eq:Astar-action}. 
A normalized cocycle $\theta\in C(F, A)$ can be expressed in matrix or vector form, and the affine action naturally decomposes into a linear part and a constant part.

Let $F=\{x_1,\dots,x_n\}$ with $x_1=1$. 
The restriction of $\theta$ to pairs $(x_i,x_j)$ with $i,j\ge 2$ yields the $(n-1) \times (n-1)$
matrix with entries in $\mathbb{F}_p^d$. Listing the rows consecutively produces a vector 
\[
\mathbf{t}= (t_{\ell})_{\ell=1}^{d(n-1)^{2}} \in \mathbb{F}_{p}^{d{(n-1)^{2}}};
\qquad
(t_{\ell},t_{\ell+1},\ldots,t_{\ell+d-1}) = \theta(x_{i},x_{j})
\]
Here each $t_\ell \in \mathbb{F}_{p}$ is one coordinate of one of the $d$-tuples $\theta(x,y)$ and have length $d(n-1)^2$.
The correspondence between a matrix entry $(i,j)$ and a vector coordinate $\ell$ is
\[
\ell=(i-2)(n-1)+(j-1)
\]

For $(\alpha,\beta,\tau)\in\mathrm{Aut}^{\!*}(F,A)$, the linear part of the affine action sends
\[
\theta(x_{i},x_{j}) \longmapsto
\theta(x_{i}^{\alpha^{-1}},x_{j}^{\alpha^{-1}})^{\beta}
\]
Thus, if $\ell$ corresponds to $(i,j)$ and $\ell'$ corresponds to $(x^{\alpha^{-1}},y^{\alpha^{-1}})$, where $x_i=x_i^{\alpha^{-1}}$ and $y_j=y_{j}^{\alpha^{-1}}$, then
\[
\mathbf{t}'[\ell]=\bigl(\mathbf{t}[\ell']^{\beta}\bigr)
\]

A matrix represents this linear transformation as
\[
\mathbf{A}\in\mathrm{Mat}_{d(n-1)^{2}}(\mathbb{F}_{p}),
\]
which depends only on $\alpha$ and $\beta$, and satisfies
\[
\mathbf{t}' = \mathbf{t}\,\mathbf{A}
\]

The constant term contributed by $\tau$, for $\tau:F\to A=\mathbb{F}_p^d$ with $\tau(1)=0$, we have
\[
b_{ij}=\partial\tau(x_i,x_j)
 =\tau(x_ix_j)-\tau(x_i)-\tau(x_j)\in\mathbb{F}_p^d
\]
By listing the coordinates of all $b_{ij}$ consecutively, we obtain a vector
\[
\mathbf{b}=(b_\ell)_{\ell=1}^{d(n-1)^2}\in \mathbb{F}_p^{\,d(n-1)^2}
\]

Each $b_\ell$ is a $d$-tuples of length $d(n-1)^2$ with entries from $\mathbb{F}_{p}$.
The vector $\mathbf{b}$ depends on $\alpha$ and $\tau$ but not on $\mathbf{t}$ i.e., $\theta$ and $\mathbf{t'}$ depends on $\alpha, \beta$ and $\mathbf{t}$.

Hence, the full affine action on coordinates is
\begin{equation} \label{eq:t}
 \mathbf{t}'=\mathbf{t}\mathbf{A}+\mathbf{b} \tag{a}
\end{equation}
where $\mathbf{A}$ describes the linear transformation determined by the permutation of matrix entries induced by $\alpha$ together with the application of $\beta$ to each coordinate in $\mathbb{F}_{p}$. Embedding the affine transformation in homogeneous coordinates gives the block matrix.
\[
\begin{bmatrix}
\mathbf{A} & 0 \\[1mm]
\mathbf{b} & 1
\end{bmatrix}
\]
so that, 
\[
[\mathbf{x}\;1]
\begin{bmatrix}
\mathbf{A} & 0 \\[1mm]
\mathbf{b} & 1
\end{bmatrix}
=
[\mathbf{xA}+\mathbf{b}\;1]
\]
This matrix, of size $\bigl(d(n-1)^2+1\bigr)\times\bigl(d(n-1)^2+1\bigr)$ over $\mathbb{F}_{p}$, provides the linear representation of $\mathrm{Aut}^{\!*}(F,A)$.

To count the number of orbits of affine automorphism group $G=\mathrm{Aut}^{\!*}(F,A)$, we have
\[
 |X/G| \;=\; \frac{1}{|G|} \sum_{g\in G} |\mathrm{Fix}(g)|
\]
where $G$ is a group acting on a finite set $X$,
In our case,
\[
X=C(F,A)\cong \mathbb{F}_{p}^{\,d(n-1)^{2}}
\]
So every cocycle is represented by a coordinate vector.
\[
\mathbf{t}\in \mathbb{F}_{p}^{\,d(n-1)^{2}}.
\]

The affine action of an element $g=(\alpha, \beta, \tau) \in G$ on cocycles induces an affine transformation on coordinates, given by 
\[
\mathbf{t'} = \mathbf{t}\mathbf{A}+\mathbf{b}
\]
 
where $\mathbf{A}$ is a $d(n-1)^{2}\times d(n-1)^{2}$ matrix over $\mathbb{F}_{p}$ depending on $(\alpha,\beta)$ and $\mathbf{b}$ is a translation vector depending on $(\alpha,\tau)$.
A cocycle $\mathbf{t}$ is fixed by $g$ exactly when
\[
\mathbf{t}=\mathbf{t}\mathbf{A}+\mathbf{b}
\]
Equivalently,
\[
\mathbf{t}(\mathbf{A}-I) = -\mathbf{b}
\]
This is a system of linear equations over $\mathbb{F}_{p}$.
The coefficient matrix is $\mathbf{A}-I$ and the unknown vector is $\mathbf{t}$.
Let $r=\mathrm{rank}(\mathbf{A}-I)$, then the homogeneous system
\[
\mathbf{t}(\mathbf{A}-I)=\mathbf{0}
\]
has a solution space of dimension
\[
d(n-1)^{2}-r
\]
The affine system
\[
\mathbf{t}(\mathbf{A}-I) = -\mathbf{b}
\]
has either no solution or it has an affine space of the same dimension.
When a solution exists, the set of fixed points has size.
\[
|\mathrm{Fix}(g)| = p^{\,d(n-1)^{2}-r}.
\]
Thus, the fixed point count depends only on the rank of $\mathbf{A}-I$ and not on the particular solution.

Since the reduction of Burnside's lemma into conjugacy classes of $G$ gives 

\[
|X\,/\,G|
=
\frac{1}{|G|}
\sum_{C\in\mathrm{Conj}(G)}
|C|\cdot |\mathrm{Fix}(g_C)|,
\]
For each such representative, one computes the matrix $\mathbf{A}$ and the rank of $\mathbf{A}-I$.
If the affine system admits a solution, then
\[
|\mathrm{Fix}(g_C)| = p^{d(n-1)^{2}-\mathrm{rank}(\mathbf{A}-I)}.
\]

\end{document}